\documentclass[12pt]{article}
\usepackage{latexsym, amsmath, amsfonts, amsthm, amssymb}
\usepackage{times}
\usepackage{a4wide}
\usepackage{tikz,pgfplots}
\usepackage{mathrsfs}
\usetikzlibrary{arrows}

\def \RR {\mathbb R}

\def \CC {\mathbb C}

\def \DD {\mathbb D}
\def \MBS {\mathbb S}

\def \eps {\varepsilon}

\newtheorem{theorem}{Theorem}[section]

\newtheorem{lemma}[theorem]{Lemma}

\newtheorem{proposition}[theorem]{Proposition}
\newtheorem{corollary}[theorem]{Corollary}

\def\myffrac#1#2 in #3{\raise 2.6pt\hbox{$#3 #1$}\mkern-1.5mu\raise 0.8pt\hbox{$
		#3/$}\mkern-1.1mu\lower 1.5pt\hbox{$#3 #2$}}

\def\qed{\hfill $\vcenter{\hrule height .3mm
		\hbox {\vrule width .3mm height 2.1mm \kern 2mm \vrule width .3mm
			height 2.1mm} \hrule height .3mm}$ \bigskip}

\begin{document}

\title{The bi-Lipschitz constant of an isothermal coordinate chart}
\author{Matan Eilat}
\date{}
\maketitle

\begin{abstract}
Let $M$ be a $C^{2}$-smooth Riemannian surface. A classical theorem in differential geometry states that the Gauss curvature function $K : M \to \RR$ vanishes everywhere if and only if the surface is locally isometric to the Euclidean plane.
We give an asymptotically sharp quantitative version of this theorem with respect to an isothermal coordinate chart.
Roughly speaking, we show that if $B$ is a Riemannian disc of radius $\delta > 0$ with $\delta^{2}\sup_{B}|K| < \eps$ for some $0 < \eps < 1$, then there is an isothermal coordinate map from $B$ onto an Euclidean disc of radius $\delta$ which is bi-Lipschitz with constant $\exp(4 \eps)$.
\end{abstract}

\section{Introduction}
Let $M$ be a $C^{2}$-smooth two-dimensional Riemannian manifold. 
Gauss's Theorema Egregium states that the Gauss curvature function $K : M \to \RR$ is invariant under local isometries. 
Hence a surface that is locally isometric to the Euclidean plane must have zero curvature. 
Following Gauss's work, Minding \cite{Mi} showed that all surfaces with the same constant curvature are locally isometric. 
The flat case of Minding's theorem therefore implies that if $K \equiv 0$ then $M$ is locally isometric to the Euclidean plane.
Here we give a quantitative version of this statement, where the Gauss curvature is not identically zero but say satisfies $|K| < \eps$ for a small $\eps > 0$.

\medskip
We consider an isothermal coordinate chart $z : \overline{B} \to \delta \overline{\DD}$. 
We write $B = B_{M}(p_{0},\delta)$ for the Riemannian disc of radius $\delta > 0$ around $p_{0} \in M$ and $\overline{B}$ for its closure.
We write $\delta \DD = B_{\CC}(0,\delta)$ for the Euclidean disc of radius $\delta$ around the origin and $\delta \overline{\DD}$ for its closure. For the boundary of $\delta \DD$ we will use the notation $\delta \MBS^{1}$. 
In these coordinates the metric takes the form
$$
g = \varphi \cdot |dz|^{2},
$$ 
where $\varphi$ is a positive function, referred to as the {\it conformal factor}. The value of the conformal factor at a point represents the infinitesimal scale of the coordinate map. More precisely, for a point $p \in \overline{B}$ we have
\begin{equation}
\sqrt{\varphi(p)} = \lim_{q \rightarrow p} \frac{d_{M}(p,q)}{|z(p) - z(q)|},
\label{eq_gamma}
\end{equation}
where $d_{M}$ is the Riemannian distance function in $M$, and $|\cdot|$ is the Euclidean norm in $\CC$.
Uniform bounds on the conformal factor therefore yield corresponding bounds for the bi-Lipschitz constant of the coordinate map and vice versa, i.e.
\begin{equation}
a \leq \frac{d_{M}(p,q)}{|z(p)-z(q)|} \leq b
\quad {\forall p \neq q \in B}
\quad \text{if and only if} \quad
a^{2} \leq \varphi(p) \leq b^{2} 
\quad \forall p\in B.
\label{eq_gamma_distance_equiv}
\end{equation}

\medskip
The main result of this paper is the following theorem, which gives uniform bounds for the conformal factor in terms of the quantity $\delta^{2} \sup_{B} \left| K \right|$. To the best of our knowledge, the mathematical literature does not contain explicit, quantitative estimates of this sort.

\begin{theorem}
Let $M$ be a $C^{2}$-smooth Riemannian surface, fix $p_{0} \in M$ and let $\delta > 0$. 
Suppose that at any point $p \in B = B_{M}(p_{0},\delta)$ the injectivity radius is at least $2\delta$ and $-\kappa \leq K(p) \leq \kappa$, where $\kappa > 0$ is a constant satisfying $\delta^{2}\kappa < \pi^{2} / 4$. 
Let $z : \overline{B} \to \delta \overline{\DD}$ be an isothermal coordinate chart such that $z(p_{0}) = 0$, $z(\partial{B}) = \delta \MBS^{1}$ and whose conformal factor is $\varphi$. Then
$$
\sup_{B} \left| \log \varphi \right| \leq 
\frac{\delta^{2}\kappa}{2}
\left[1  + \frac{\pi^{2}}{4} \cdot \left( \frac{\sinh(\delta \sqrt{\kappa}) \tan(\delta \sqrt{\kappa})}{\delta^{2}\kappa}\right)^{2} \right].
$$
\label{thm_main}
\end{theorem}

\medskip
Observe that the term in the inner brackets is increasing as a function of $\delta^{2}\kappa$, tends to $1$ as $\delta^{2}\kappa \searrow 0$ and tends to infinity as $\delta^{2}\kappa \nearrow \pi^{2}/4$.
In case $\delta^{2}\kappa$ is assumed to be bounded away from $\pi^{2}/4$, the estimate may be expressed as $C \cdot \delta^{2}\kappa$ for some universal constant $C$.
We thus formulate the following corollary for simplicity and ease of use.
Note that the second part of the corollary follows from the first part together with (\ref{eq_gamma_distance_equiv}). 

\begin{corollary}
Under the assumptions of Theorem \ref{thm_main} together with the additional assumption that $\delta^{2}\kappa < \pi^{2} / 8$, we have that 
$$
\sup_{B} \left| \log \varphi \right| \leq 8 \delta^{2}\kappa,
$$
and for any two distinct points $p, q \in B$ we have
$$
\exp(-4\delta^{2}\kappa)
\leq \frac{d_{M}(p,q)}{| z(p) - z(q)|} \leq 
\exp(4\delta^{2}\kappa).
$$
\label{cor_main_easy}
\end{corollary}

\medskip
The isothermal coordinate chart has several advantages over other coordinate charts, for example in terms of regularity.
As a two-dimensional special case of harmonic coordinates, the metric in this chart is as smooth as it can be. This is according to Deturck and Kazdan \cite{DK}, who showed that if a metric is of class $C^{k,\alpha}$ in some coordinate chart then it is also of class $C^{k,\alpha}$ in harmonic coordinates. 
However, this property does not hold for other coordinate charts. According to \cite{DK} the metric in normal coordinates is guaranteed to be only of class $C^{k-2, \alpha}$. In addition they give an example of a metric that is of class $C^{k,\alpha}$ in harmonic coordinates but not of class $C^{k-1}$ in normal coordinates.

\medskip
The bi-Lipschitz constant derived in Corollary \ref{cor_main_easy} for the isothermal coordinate chart is optimal, up to some multiplicative constant in the exponent. 
Indeed, Milnor \cite{Mil} showed that in the case of a spherical cap, the normal coordinate chart has the minimal ``distortion'', which behaves asymptotically like $\delta^{2}\kappa / 6$ for small values of $\delta^{2}\kappa$.
The distortion of a coordinate chart is defined as the logarithm of the ratio between the maximal scale and the minimal scale, where the maximal scale is given by the supremum of the ratio between the Euclidean distance under the coordinate chart and the Riemannian distance in the surface, and the minimal scale is given by the infimum. 
Hence we see that our estimates are asymptotically sharp, and there is essentially no disadvantage in considering an isothermal coordinate chart in the aspect of local metric distortion.

\medskip
The remainder of this paper is devoted to the proof of Theorem \ref{thm_main}, and from now on we work under its assumptions.
We let $M$ be a $C^{2}$-smooth Riemannian surface, fix a point $p_{0}\in M$ and $\delta > 0$, and consider the Riemannian disc $B = B_{M}(p_{0},\delta)$.
We assume that the injectivity radius of all points of $B$ is at least $2\delta$, and that the Gauss curvature function $K : M \to \RR$ satisfies $-\kappa \leq K \big|_{B} \leq \kappa$ for some constant $\kappa > 0$ such that $\delta^{2}\kappa < \pi^{2} / 4$.

\medskip
{\it Acknowledgements.} I would like to express my deep gratitude to my advisor, Prof. Bo'az Klartag, for his continuous guidance, assistance and encouragement.
Supported by the Adams Fellowship Program of the Israel Academy of Sciences and Humanities and the Israel Science Foundation (ISF).

\section{Preliminaries}
\label{sec_2}

\medskip
A set $U \subseteq M$ is called strongly convex if its closure $\overline{U}$ has the property that for any $p,q \in \overline{U}$ there is a unique minimizing geodesic in $M$ from $p$ to $q$, and the interior of this geodesic is contained in $U$. 
Under our assumption that the injectivity radius at each point of the Riemannian disc $B$ is at least $2\delta$ and furthermore
$\delta^{2} \kappa < \pi^{2} / 4$, we have that $B$ is strongly convex according to the Whitehead theorem \cite{Wh}.

\medskip
Since $M$ is a $C^{2}$-smooth Riemannian surface, around any point $p \in M$ there exists an isothermal coordinate chart, i.e. there is
neighborhood $U$ containing $p$ and a coordinate chart $w = x + iy : U \to \CC$, such that the metric in these coordinates is of the form
$$
g = \lambda \left| dw \right|^{2} = \lambda ( dx^{2}+dy^{2} ),
$$
where $\lambda$ is a positive function which is called the conformal factor. 
Thanks to Deturck and Kazdan \cite{DK} we know that $\lambda$ is a $C^{2}$-smooth function.
The Gauss curvature function $K$ is given by Liouville's equation
\begin{equation}
K = -\frac{\Delta \left( \log \lambda \right)}{2\lambda},
\label{eq_curvature_formula}
\end{equation}
where $\Delta$ is the usual Laplace operator with respect to the coordinate map $w = x+iy$.
For the Laplace-Beltrami operator we use the notation $\Delta_{M}$. Under the isothermal coordinate chart we have 
\begin{equation}
\Delta_{M} = \lambda^{-1}\Delta.
\label{eq_deltas}
\end{equation}

\medskip
An upper semi-continuous function $u: M \to \RR \cup \{ -\infty \}$ is {\it subharmonic} if each $p \in M$ belongs to an isothermal coordinate chart in which $u$ is subharmonic. 
The function $u$ is called superharmonic if $-u$ is subharmonic. A function that is both subharmonic and superharmonic is harmonic. 
Since the conformal factor $\lambda$ is positive, it follows from (\ref{eq_deltas}) that a smooth function $u : M \to \RR$ is subharmonic if and only if $\Delta_{M}u \geq 0$. 
The {\it maximum principle} for subharmonic functions states that for a bounded, connected, open set $\Omega \subseteq \CC$ and a subharmonic function $u$ on $\Omega$, if $\limsup u(z) \leq 0$ as $z \rightarrow \partial{\Omega}$ then $u(z) \leq 0$ for all $z \in \Omega$.
See \cite{Ga, Hay} for more details.

\medskip
The uniformization theorem states that every simply-connected Riemann surface is conformally equivalent to a disc, to the complex plane, or to the Riemann sphere. See \cite{Ah} for information.
It follows from the uniformization theorem that the Riemannian disc $B$ is conformally equivalent to the unit disc $\DD$.
To see this, fix $p \in \overline{B}$ and consider a normal polar coordinate chart $(\rho,\theta): \overline{B} \setminus \{p\} \to (0,2\delta] \times [0,2\pi)$. The metric in these coordinates is given by 
$$
g = d\rho^{2} + \phi^{2}d\theta^{2},
$$
where $\phi$ is a positive function satisfying the Jacobi equation. In particular, if we fix $\theta \in [0,2\pi)$ and define $\phi_{\theta}(\rho) = \phi(\rho, \theta)$ then
$$
\lim_{\rho \rightarrow 0} \phi_{\theta}(\rho) = 0, 
\qquad
\lim_{\rho \rightarrow 0}\phi_{\theta}'(\rho) = 1
\qquad \text{and} \qquad
\phi_{\theta}''(\rho) + K(\rho, \theta)\cdot\phi_{\theta}(\rho) = 0.
$$
The Laplacian of the distance function $r$ from $p$ is given by
$ \Delta_{M} r(\rho,\theta) = \phi_{\theta}'(\rho) / \phi_{\theta}(\rho)$.
It follows from a well-known comparison argument that under our assumptions we have 
\begin{equation}
\frac{\sin(\rho \sqrt{\kappa})}{\sqrt{\kappa}} \leq
\phi(\rho,\theta) \leq 
\frac{\sinh(\rho \sqrt{\kappa})}{\sqrt{\kappa}},
\label{eq_1856}
\end{equation}
and
\begin{equation}
\sqrt{\kappa} \cot(\rho \sqrt{\kappa})
\leq \Delta_{M} r(\rho, \theta) \leq 
\sqrt{\kappa} \coth(\rho \sqrt{\kappa}).
\label{eq_distance_laplacian_comparison}
\end{equation}
See \cite{Pet} for more details and comparison results.
As a consequence of (\ref{eq_distance_laplacian_comparison}) we obtain that the distance function from the center $p_{0}$ of the disc $B$ is subharmonic. Moreover, this function is clearly non-constant and bounded from above.  According to Liouville's theorem, such a function does not exist on the entire complex plane $\CC$. 
Since $B$ is not compact and therefore cannot be conformally equivalent to the sphere, we conclude that $B$ must be conformally equivalent to the disc $\DD$.
The conformal equivalence between $B$ and $\DD$ may be given by an isothermal coordinate chart 
$$
w = x+iy : \overline{B} \rightarrow \overline{\DD}
\qquad  \text{such that} \quad
w(p_{0}) = 0
\quad \text{and} \quad 
w(\partial{B}) = \MBS^{1}.
$$
To explain the boundary behavior, consider a slightly larger Riemannian disc $B_{\eps} := B_{M}(p_{0}, \delta + \eps)$. For a sufficiently small $\eps > 0$, the Riemannian disc $B_{\eps}$ will also be conformally equivalent to the unit disc $\DD$ via some mapping $w_{\eps}: B_{\eps} \to \DD$ with $w(p_{0}) = 0$. 
The set $U := w_{\eps}(B) \subseteq \DD$ is a simply-connected open set.
Hence by the Riemann mapping theorem there exists a conformal mapping $f : U \to \DD$ with $f(0) = 0$. Moreover, the set $U$ has a $C^{\infty}$-smooth boundary, being the level set of the regular value $\delta$ of the Riemannian distance function from $p_{0}$.
Hence the map $f$ extends to a diffeomorphism between $\overline{U}$ and $\overline{\DD}$, due to Painlev\'{e} \cite{Pai} and Kellogg \cite{Kel}. Thus $w := f \circ w_{\eps}|_{B}$ extends continuously to a function on the closure $\overline{B}$.
Moreover, if $\lambda$ and $\lambda_{\eps}$ are the conformal factors associated with the coordinate charts $w$ and $w_{\eps}$ respectively, then $\lambda = |(f^{-1})'|^{2} \cdot (\lambda_{\eps}|_{U} \circ f^{-1})$  and we see that $\lambda$ extends continuously to a function on $\overline{\DD}$.

\medskip
To avoid confusion, we will stick to the following notation throughout:
the coordinate map $z: \overline{B} \to \delta \overline{\DD}$ will be used for isothermal coordinates with $z(p_{0})=0$ and $z(\partial{B})=\delta \MBS^{1}$, as described in Theorem \ref{thm_main}. The corresponding conformal factor will be denoted by $\varphi$.
The coordinate map $w = \delta^{-1}z: \overline{B} \to \overline{\DD} $ will be used for isothermal coordinates with $w(p_{0}) = 0$ and $w(\partial{B}) = \MBS^{1}$, and the corresponding conformal factor will be denoted by $\lambda$.
We thus have the relations
$$
z = \delta w 
\qquad \text{and} \qquad
\varphi(z) = \delta^{-2} \lambda(w) .
$$

\medskip
The barriers we introduce in Section \ref{sec_3} are defined on a punctured disc, with a logarithmic singularity at the point in which they are not defined.
Let $p \in \overline{B}$ and suppose that $u: \overline{B} \setminus \{p\} \to \RR$ is a continuous function that is subharmonic on $B \setminus \{p\}$.
We say that $u$ has a {\it logarithmic pole} at $p$ if 
$$
\limsup_{q \to p} \left( u(q) + \log |w(p) - w(q)| \right) < \infty.
$$
In this case the function $u(q) + \log |w(p) - w(q)|$ extends uniquely to a subharmonic function on $B$.
When we use the maximum principle for a function of this sort, we mean that we apply it with respect to the above-mentioned extension. Note that every assertion on subharmonic functions, e.g. the maximum principle, naturally gives rise to a symmetric assertion on superharmonic functions.

\medskip
The Riemannian disc $B$ is a $\text{CAT}(\kappa)$ space, as it is strongly convex and the Gauss curvature function in $B$ is bounded from above by $\kappa$.
According to Alexandrov \cite{Alek} geodesic triangles in a $\text{CAT}(\kappa)$ space satisfy certain comparison properties with respect to corresponding geodesic triangles in the model space of constant curvature $\kappa$.
See \cite{Alex, BH} for information about $\text{CAT}(k)$ spaces.
One such comparison property that we will use later is stated below.
Write $S_{\kappa}$ for the two-dimensional sphere of radius $1/\sqrt{\kappa}$, and write $d_{\kappa}$ for the induced metric on $S_{\kappa}$.
For $p,q,r \in \overline{B}$ we write $\Delta = \Delta(p,q,r) \subseteq \overline{B}$ for the geodesic triangle whose edges are the minimizing geodesic segments $[p,q],[q,r],[r,p]$ connecting their corresponding endpoints.
We let $L(\Delta) = d_{M}(p,q) + d_{M}(q,r) + d_{M}(p,r)$ denote its perimeter.

\begin{proposition}
For any geodesic triangle $\Delta = \Delta(p,q,r) \subseteq \overline{B}$ with perimeter $L(\Delta) < 2\pi / \sqrt{\kappa}$, with $p \neq q$ and $p \neq r$, if $\theta$ denotes the angle between $[p,q]$ and $[p,r]$ at $p$, and if $\Delta_{0} = \Delta(p_{0},q_{0},r_{0})$ is a geodesic triangle in $S_{\kappa}$ with 
$ d_{\kappa}(p_{0},q_{0}) = d_{M}(p,q) $,
$ d_{\kappa}(p_{0},r_{0}) = d_{M}(p,r) $,
and angle at $p_{0}$ equal to $\theta$, then $d_{M}(q,r) \geq d_{\kappa}(q_{0},r_{0})$.
\label{prop_cat}
\end{proposition}

\section{The harmonic barriers}
\label{sec_3}

Given the conformal equivalence $w : B \to \DD$ we define Green's function $G : B \setminus \{p_{0}\} \to \RR$ by $G = -\log |w|$. Observe that $G$ is a positive function with $G(p) \longrightarrow 0$ as $p \to \partial B$, and $G(p) \longrightarrow \infty$ as $p \to p_{0}$.
Moreover, the function $G$ is harmonic on $B \setminus \{ p_{0} \}$ with a logarithmic pole at $p_{0}$.
Our choice of harmonic barriers is motivated by Green's functions in the extremal cases, which are the surfaces of constant curvature $\kappa$ and $-\kappa$.
For a point $p \in \overline{B}$ we define the functions $S_{p}, H_{p}: \overline{B}\setminus \{p \} \to \RR$ as follows:
$$
S_{p}(q) = 
-\log \left( \tan \left( \frac{d_{M}(q,p)\sqrt{\kappa}}{2} \right)
\right)
\quad \text{and} \quad
H_{p}(q) = 
-\log \left( \tanh \left( \frac{d_{M}(q,p)\sqrt{\kappa}}{2} \right)
\right).
$$
Under our assumption that $\delta^{2}\kappa < \pi^{2} / 4$ we see that these functions are well-defined.
The relevant properties of $S_{p}$ and $H_{p}$ are summarized in the following proposition.

\begin{proposition}
For any $p \in \overline{B}$ we have that:
\begin{itemize}
\item[(i)] $S_{p}$ is superharmonic on $ B \setminus \{ p \}$ and $H_{p}$ is subharmonic on $ B \setminus \{ p \}$.
\item[(ii)] Both $S_{p}$ and $H_{p}$ are monotonically decreasing with the distance from $p$.
\item[(iii)] Both $S_{p}$ and $H_{p}$ have a logarithmic pole at $p$.
\end{itemize}
\label{prop_barriers_properties}
\end{proposition}

\begin{proof} 
Write $r:\overline{B} \rightarrow \RR $ for the distance function from $p$. Differentiating $S_{p}$ and $H_{p}$ with respect to $r$ gives $\partial_{r} S_{p} =  -\sqrt{\kappa}/\sin(r \sqrt{\kappa})$ and $\partial_{r} H_{p} = -\sqrt{\kappa}/\sinh(r\sqrt{\kappa})$.
Since $0 < r \leq 2\delta$ and $\delta \sqrt{\kappa} < \pi / 2$ we see that (ii) holds true. Differentiating once more and using the chain rule we obtain that
$$
\Delta_{M} S_{p} =
\frac{\kappa \cot(r\sqrt{\kappa})}{\sin(r\sqrt{\kappa})} - \frac{\sqrt{\kappa} \cdot \Delta_{M} r}{\sin(r\sqrt{\kappa})}
\quad \text{and} \quad
\Delta_{M} H_{p} = \frac{\kappa \coth(r\sqrt{\kappa})}{\sinh(r\sqrt{\kappa})} - \frac{\sqrt{\kappa} \cdot \Delta_{M} r}{\sinh(r\sqrt{\kappa})}.
$$
Therefore it follows from (\ref{eq_distance_laplacian_comparison}) that 
$\Delta_{M} S_{p} \leq 0 \leq \Delta_{M} H_{p}$, thus proving (i). 
Using (\ref{eq_gamma}) together with the fact that $ \lim_{x \to 0} \tan(ax)/x = a$ we have that 
\begin{align*}
\lim_{q \rightarrow p}
\left( S_{p}(q) + \log |w(p)-w(q)|  \right) &=
\lim_{q \rightarrow p}
\left( S_{p}(q) + \log(r(q)) \right) 
\\ & -
\lim_{q \rightarrow p} \left( \log(r(q)) - \log |w(p)-w(q)|  \right) 
= -\frac{1}{2} \log \left( \frac{\kappa \lambda(p)}{4} \right).
\end{align*}
Since $ \lim_{x \to 0} \tanh(ax)/x = a$ as well, a similar equality holds for $H_{p}$ so that (iii) holds true and the proof is completed.
\end{proof}

The following proposition gives bounds for the conformal factor $\lambda$ at the origin and on the boundary.
Observe that the upper bound at the origin is given by the value of the spherical conformal factor at the origin, and the lower bound is given by the hyperbolic one.
On the other hand, on the boundary the situation is reversed, meaning that the upper bound is given by the hyperbolic value and the lower bound by the spherical one.
This might indicate why the bounds on the entire disc are not trivially given by the extremal cases.

\begin{proposition}
We have that
\begin{equation}
\frac{4}{\kappa} 
\tanh^{2} \left( \frac{\delta\sqrt{\kappa}}{2} \right)
\leq \lambda(p_{0}) \leq 
\frac{4}{\kappa} 
\tan^{2} \left( \frac{\delta\sqrt{\kappa}}{2} \right),
\label{eq_1947}
\end{equation}
and
\begin{equation}
\frac{\sin^{2}(\delta \sqrt{\kappa})}{\kappa}
\leq \lambda \big|_{\partial{B}} \leq
\frac{\sinh^{2}(\delta \sqrt{\kappa})}{\kappa}.
\label{eq_2014}
\end{equation}
\label{prop_2140}
\end{proposition}

\begin{proof}
Write $G = -\log |w|$ for the associated Green's function.
Then 
\begin{equation}
S_{p_{0}}(q) + \log(\tan(\delta \sqrt{\kappa}/2)) = 
H_{p_{0}}(q) + \log(\tanh(\delta \sqrt{\kappa} / 2)) = 
G(q) = 0 
\quad \text{for all }q\in \partial{B}.
\label{eq_1914}
\end{equation} 
According to Proposition \ref{prop_barriers_properties} we have that  $S_{p_{0}}$ is superharmonic and $H_{p_{0}}$ is subharmonic. Applying the maximum principle we obtain that
\begin{equation}
H_{p_{0}}(p) + \log(\tanh(\delta \sqrt{k}/ 2)) 
\leq G(p) \leq 
S_{p_{0}}(p) + \log(\tan(\delta \sqrt{k}/ 2))
\qquad \text{for all } p \in B.
\label{eq_barriers}
\end{equation}
Define $r: \overline{B} \to \RR$ by $r(p) = d_{M}(p_{0},p)$. 
Using (\ref{eq_barriers}) together with the fact that $\lim_{x \rightarrow 0 } \tan(ax) / x = \lim_{x \rightarrow 0} \tanh(ax) / x = a$ we obtain that for any $p \in B$,
$$
\frac{2}{\sqrt{\kappa}} 
\tanh \left( \frac{\delta \sqrt{\kappa}}{2} \right) 
\leq 
\lim_{p \rightarrow p_{0}}\frac{r(p)}{|w(p)|} 
\leq
\frac{2}{\sqrt{\kappa}}
\tan \left( \frac{\delta \sqrt{\kappa}}{2} \right).
$$
Thus by using (\ref{eq_gamma}) we see that (\ref{eq_1947}) holds true.
The functions in (\ref{eq_1914}) are equal to zero on the boundary $\partial B$. In the disc $B$ they satisfy the inequalities given in (\ref{eq_barriers}). Hence we conclude that for any $q \in \partial{B}$ we have
$$
\frac{\sqrt{\kappa}}{\sinh(\delta \sqrt{\kappa})}
= -\partial_{r} H_{p_0}(q) 
\leq \Vert \nabla G(q) \Vert \leq 
-\partial_{r} S_{p_0}(q)
= \frac{\sqrt{\kappa}}{\sin(\delta \sqrt{\kappa})},
$$
where $\Vert \cdot \Vert$ denotes the Riemannian norm on the relevant tangent space.
Since $\lambda(q) = \Vert \nabla G(q) \Vert^{-2}$ for $q\in \partial{B}$ we see that (\ref{eq_2014}) holds true, thus completing the proof.
\end{proof}

\section{The ratio between distances}
\label{sec_4}

Given two distinct points $q_{1}$ and $q_{2}$ on the boundary $\partial{B}$, there are two paths joining them in $\partial{B}$. The length of the shorter path is the arc-distance, which we denote by $d_{\partial{B}}(q_{1},q_{2})$.
When $w : \overline{B} \to \overline{\DD}$ is an isothermal coordinate chart with $w(p_{0}) = 0$ and $w(\partial{B}) = \MBS^{1}$, we denote the Euclidean arc-distance between $w(q_{1})$ and $w(q_{2})$ in $\MBS^{1}$ by $d_{\MBS^{1}}(w(q_{1}),w(q_{2}))$.
The infinitesimal ratio between $d_{\partial B}$ and $d_{\MBS^{1}}$ is given by the square-root of conformal factor $\lambda$. Hence by using (\ref{eq_2014}) from Proposition \ref{prop_2140} we see that
\begin{equation}
\frac{\sin(\delta\sqrt{\kappa})}{\sqrt{\kappa}} \leq \frac{d_{\partial{B}}(q_{1},q_{2})}{d_{\MBS^{1}}(w(q_{1}),w(q_{2}))} 
\leq \frac{\sinh(\delta\sqrt{\kappa})}{\sqrt{\kappa}}.
\label{eq_2140}
\end{equation}

\medskip
Write $\angle(q_{1},q_{2})$ for the angle between the geodesic segments in $B$ connecting $q_{1}$ and $q_{2}$ with the center $p_{0} \in B$.
The infinitesimal ratio between $d_{\partial B}$ and $\angle$ is given by the Jacobian of the normal polar coordinate chart.
It therefore follows from the comparison argument (\ref{eq_1856}) that
\begin{equation}
\frac{\sin(\delta\sqrt{\kappa})}{\sqrt{\kappa}}
\leq 
\frac{d_{\partial{B}}(q_{1},q_{2})}{\angle(q_{1},q_{2})} 
\leq
\frac{\sinh(\delta\sqrt{\kappa})}{\sqrt{\kappa}}.
\label{eq_2151}
\end{equation}

\medskip
For the Euclidean unit disc we have the following inequalities bounding the ratio between the arc-distance and the chordal distance of two distinct points $w(q_{1}), w(q_{2})\in \MBS^{1}$,
\begin{equation}
1 \leq \frac{d_{\MBS^{1}}(w(q_{1}),w(q_{2}))}{|w(q_{1}) - w(q_{2})|} \leq \frac{\pi}{2}.
\label{eq_0229}
\end{equation}
A similar inequality holds in the more general case of a surface that satisfies our assumptions, as we show in the following lemma. The qualitative statement of the lemma is that the maximum of the ratio between the angle at the origin and the Riemannian distance of two points on the boundary is attained for antipodal points.
Our proof relies on the comparison property stated in Proposition \ref{prop_cat}, even though the statement of the lemma does not involve the curvature bounds. 
We therefore believe that a more direct proof could probably be found.

\begin{lemma} 
For any two distinct points $q_{1}, q_{2} \in \partial{B}$ we have
$$
\frac{\angle(q_{1},q_{2})}{d_{M}(q_{1},q_{2})} \leq
\frac{\pi}{2\delta} .
$$
\label{lem_ratio_angle_arc}
\end{lemma}

\begin{proof} 
Consider the sphere of constant curvature $\kappa$ and a geodesic triangle on this sphere with edges of lengths $a$, $b$ and $c$. Let $\gamma$ be the angle opposite to the edge of length $c$. According to the spherical law of cosines we have that
\begin{equation}
\cos(c\sqrt{\kappa}) = 
\cos(a\sqrt{\kappa}) \cos(b\sqrt{\kappa}) +
\sin(a\sqrt{\kappa}) \sin(b\sqrt{\kappa}) \cos(\gamma).
\label{eq_1449}
\end{equation}
Since $d_{M}(p_{0},q_{1}) = d_{M}(p_{0},q_{2}) = \delta$ the perimeter of the geodesic triangle $\Delta(p_{0},q_{1},q_{2})$ in $M$ is at most $4\delta$, which is less than $2\pi / \sqrt{\kappa}$ according to our assumptions. 
It therefore follows from Proposition \ref{prop_cat} and (\ref{eq_1449}) that
$$
\frac{d_{M}(q_{1},q_{2})}{\angle(q_{1},q_{2})} \geq 
\frac{\cos^{-1}(\cos^{2}(\delta\sqrt{\kappa}) +
\sin^{2}(\delta\sqrt{\kappa}) \cos(\angle(q_{1},q_{2})))}{\sqrt{\kappa} \cdot \angle(q_{1},q_{2})}.
$$
According to Lemma \ref{lem_attains_inf} from the appendix, the infimum of the right-hand side is attained when $\angle(q_{1},q_{2})= \pi$. Moreover, since $\cos(2\theta) = \cos^{2}(\theta) - \sin^{2}(\theta)$ and $\delta \sqrt{\kappa} < \pi / 2$ we see that 
$$
\frac{d_{M}(q_{1},q_{2})}{\angle(q_{1},q_{2})} \geq \frac{2\delta \sqrt{\kappa}}{\sqrt{\kappa}\cdot \pi} = \frac{2\delta}{\pi},
$$
and the proof is completed.
\end{proof}

\medskip
The following corollary combines the inequalities above in order to bound the ratio between the Riemannian distance and the Euclidean distance for two points on boundary. 

\begin{corollary}
For any two distinct points $q_{1}, q_{2} \in \partial{B}$ we have that
$$
\frac{2 \delta \sin (\delta\sqrt{\kappa})}{\pi \sinh(\delta\sqrt{\kappa})} 
\leq 
\frac{d_{M}(q_{1},q_{2})}{|w(q_{1})-w(q_{2})|}
 \leq
\frac{\pi \sinh(\delta\sqrt{\kappa})}{2 \sqrt{\kappa}} .
$$
\label{cor_2331}
\end{corollary}

\begin{proof}
For two distinct points $q_{1},q_{2} \in \partial B$ we have that
$$
\frac{d_{M}(q_{1},q_{2})}{|w(q_{1})-w(q_{2})|} \leq
\frac{d_{\partial B}(q_{1},q_{2})}{|w(q_{1})-w(q_{2})|} =
\frac{d_{\partial B}(q_{1},q_{2})}{d_{\MBS^{1}}(w(q_{1}),w(q_{2}))} \cdot
\frac{d_{\MBS^{1}}(w(q_{1}),w(q_{2}))}{|w(q_{1})-w(q_{2})|}.
$$
The upper bound therefore follows from (\ref{eq_2140}) together with (\ref{eq_0229}).
For the lower bound we write 
$$
\frac{d_{M}(q_{1},q_{2})}{|w(q_{1})-w(q_{2})|} = 
\frac{d_{M}(q_{1},q_{2})}{\angle(q_{1},q_{2})} \cdot
\frac{\angle(q_{1},q_{2})}{d_{\partial{B}}(q_{1},q_{2})} \cdot
\frac{d_{\partial{B}}(q_{1},q_{2})}{d_{\MBS^{1}}(w(q_{1}),w(q_{2}))} \cdot
\frac{d_{\MBS^{1}}(w(q_{1}),w(q_{2}))}{|w(q_{1}) - w(q_{2})|},
$$
and we see that it follows from Lemma \ref{lem_ratio_angle_arc} together with (\ref{eq_2140}), (\ref{eq_2151}) and (\ref{eq_0229}).
\end{proof}

\medskip
In order to bound the ratio between the distances for any two points in $B$ we will successively use the maximum principle for the barriers $S_{p}$ and $H_{p}$ from Section \ref{sec_3} with respect to different points $p \in \overline{B}$.
The boundary inequalities will be met by adding constants whose values are given by virtue of Corollary \ref{cor_2331}. 

\begin{lemma}
Fix $q_{0} \in \partial B$. Then for any $q \in \partial{B} \setminus \{ q_0 \}$ we have
$$
H_{q_{0}}(q) + C_{h}
\leq - \log |w(q) - w(q_{0})| \leq 
S_{q_{0}}(q) + C_{s},
$$
where
$$
C_{h} = \log \left( \frac{\sin (\delta\sqrt{\kappa})}{\pi \cosh(\delta\sqrt{\kappa})} \right)
\qquad \text{and} \qquad
C_{s} = \log \left( \frac{\pi \sinh(\delta\sqrt{\kappa}) \tan ( \delta\sqrt{\kappa})}{4\delta\sqrt{\kappa}}\right).
$$
\label{lem_0016}
\end{lemma}

\begin{proof}
Define $r(p) = d_{M}(p, q_{0})$. According to Corollary \ref{cor_2331} for any $q \in \partial{B} \setminus \{q_{0}\}$,
\begin{equation}
\log \left( \frac{2 \delta\sin (\delta\sqrt{\kappa})}{\pi \sinh( \delta\sqrt{\kappa})} \right)\leq 
\log (r(q)) - \log |w(q) - w(q_{0})| \leq
\log \left( \frac{\pi \sinh(\delta\sqrt{\kappa})}{2 \sqrt{\kappa}} \right) .
\label{eq_2333}
\end{equation}
By Lemma \ref{lem_0048} from the appendix, for any $a > 0$ the function $x / \tanh(ax)$ is increasing on $(0, \infty)$. Since $0 < r(q) \leq 2\delta$ we have that 
$$
\log (r(q)) - 
\log \left(\tanh \left( r(q) \sqrt{\kappa}/2 \right)\right) \leq 
\log (2\delta) - \log (\tanh(\delta\sqrt{\kappa})) .
$$
Hence it follows from (\ref{eq_2333}) that for any $q \in \partial{B} \setminus \{q_{0}\}$,
$$
H_{q_{0}}(q) \leq 
-\log(r(q)) - \log(\tanh(\delta\sqrt{\kappa})) + \log(2\delta)
 \leq
- C_{h} -\log |w(q) - w(q_{0})| .
$$
By Lemma \ref{lem_0048} we also have that for any $a > 0$ the function $x / \tan(ax)$ is decreasing on $(0, \pi/2a)$. Since $0 < r(q) \leq 2\delta < \pi / \sqrt{\kappa}$ we obtain that
$$
\log (r(q)) - 
\log \left(\tan \left( r(q) \sqrt{\kappa}/2 \right)\right) \geq 
\log (2\delta) - \log (\tan(\delta\sqrt{\kappa})).
$$
Thus by (\ref{eq_2333}) we have that for any $q \in \partial{B} \setminus \{q_{0}\}$,
$$
S_{q_{0}}(q) \geq  -\log(r(q)) - \log(\tan(\delta\sqrt{\kappa})) + \log(2\delta) \geq
- C_{s} -\log |w(q) - w(q_{0})|,
$$
and the proof is completed.
\end{proof}

\begin{proposition}
For any two distinct points $p,q \in B$ we have
$$
\frac{2}{\pi} \cdot \frac{\sin(\delta \sqrt{\kappa})}{\sqrt{\kappa} \cosh(\delta \sqrt{\kappa})}
\leq \frac{d_{M}(p,q)}{|w(p)-w(q)|} \leq 
\frac{\pi}{2} \cdot \frac{\sinh(\delta\sqrt{\kappa}) \tan ( \delta\sqrt{\kappa})}{\delta\kappa} .
$$
\label{prop_distance_ratio} 
\end{proposition}

\begin{proof}
Fix $q_{0}\in \partial{B}$ and consider the functions $h,s: \overline{B} \setminus \{ q_{0} \} \to \RR$ given by
$$
h(p) = H_{q_{0}}(p) +\log |w(p)-w(q_{0})| + C_{h} 
\quad \text{and} \quad
s(p) = S_{q_{0}}(p) +\log |w(p)-w(q_{0})| + C_{s} ,
$$
where $C_{h}$ and $C_{s}$ are as defined in Lemma \ref{lem_0016}. According to this lemma we have that 
$$
h(q) \leq 0 \leq s(q) 
\qquad \text{for all } q\in \partial{B} \setminus \{ q_{0} \}.
$$
Moreover, by Proposition \ref{prop_barriers_properties} we see that $h$ is subharmonic on $B$, that $s$ is superharmonic on $B$, and that both functions are well-defined and continuous at $q_{0}$.
Hence by the maximum principle we obtain that for any $p\in B$,
\begin{equation}
H_{q_{0}}(p) + C_{h}
\leq -\log |w(p)-w(q_{0})| \leq 
S_{q_{0}}(p) + C_{s}.
\label{eq_0234}
\end{equation}
Fix $p_{1} \in B$. Observe that for $q \in \partial B$ we have the symmetries $H_{q}(p_{1}) = H_{p_{1}}(q)$ and $S_{q}(p_{1}) = S_{p_{1}}(q)$.
Hence it follows from (\ref{eq_0234}) that for any $q\in \partial{B}$,
$$
H_{p_{1}}(q) + C_{h} 
\leq -\log |w(p_{1})-w(q)| \leq 
S_{p_{1}}(q) + C_{s}.
$$
Applying the maximum principle again we obtain that for any $p_{2} \in B \setminus \{ p_{1} \}$,
$$
H_{p_{1}}(p_{2}) + C_{h} 
\leq -\log |w(p_{1})-w(p_{2})| \leq
S_{p_{1}}(p_{2}) + C_{s}.
$$
The result is therefore obtained by using the additional fact that $\tanh(x) < x < \tan(x)$ for $0 < x < \pi / 2$.
\end{proof}

Using Proposition \ref{prop_2140} and Proposition \ref{prop_distance_ratio} we are able to prove Theorem \ref{thm_main}. 
The only missing ingredient is the following lemma which is a rather simple application of the maximum principle.

\begin{lemma}
For $u \in C^{2}(\delta \DD)\cap C^{0}(\delta \overline{\DD})$ we have
$$
\sup_{\delta \DD} |u| \leq 
\sup_{\delta \MBS^{1}} |u| + \frac{\delta^{2}}{4} \cdot \sup_{\delta \DD} |\Delta u| .
$$
\label{lem_go_to_boundary}
\end{lemma}

\begin{proof}
In case $\sup_{\delta \DD}| \Delta u |$ is infinite, the inequality is trivial. Otherwise, define $v : \delta \overline{\DD} \to \RR$ by
$$
v(z) = \sup_{\delta \MBS^{1}} |u| + \frac{\delta^{2} - |z|^{2}}{4} \cdot  \sup_{\delta \DD}|\Delta u|.
$$
Since $ \Delta |z|^{2} = 4 $ we have that 
$\Delta v = - \sup_{\delta \DD}|\Delta u|$. 
Therefore $\Delta(v-u) \leq 0$ and $v-u$ is superharmonic in $\delta \DD$. 
Moreover, for any $\zeta \in \delta \MBS^{1}$ we have that
$ (v-u)(\zeta) \geq 0$.
Hence by applying the maximum principle we obtain that for any $z \in \delta \DD$,
$$
u(z) \leq v(z) \leq
\sup_{\delta \MBS^{1}} |u| + \frac{\delta^{2}}{4} \cdot \sup_{\delta \DD}|\Delta u| .
$$ 
A similar inequality holds for $-u$ as well, and therefore we obtain the result of lemma.
\end{proof}

\begin{proof} [Proof of Theorem \ref{thm_main}]
Recall that $z = \delta w$ and that the corresponding conformal factors admit the relation $\varphi(z) = \delta^{-2} \lambda(w)$.
According to Proposition \ref{prop_distance_ratio} together with (\ref{eq_gamma_distance_equiv}) we have that
$$
\varphi \leq 
\frac{\pi^{2}}{4} \cdot \left( \frac{\sinh(\delta\sqrt{\kappa}) \tan ( \delta\sqrt{\kappa})}{\delta^{2}\kappa} \right)^{2} .
$$ 
Proposition \ref{prop_2140} together with the fact that
$ \log ( \sin^{2}(x) / x^{2} ) > - x^{2} / 2 $
and 
$ \log ( \sinh^{2}(x) / x^{2} ) < x^{2}/2 $
for any $0 < x < \pi / 2$,
which follows from Lemma \ref{lem_some_estimates} in the appendix, shows that
$$
-\frac{\delta^{2}\kappa}{2}
\leq \log \left( \varphi \big|_{\partial{B}} \right) \leq
\frac{\delta^{2}\kappa}{2} .
$$
Therefore according to Lemma \ref{lem_go_to_boundary} and Liouville's equation (\ref{eq_curvature_formula}) we obtain that 
$$
\sup_{B} \left| \log \varphi \right| \leq 
\sup_{\partial{B}} \left| \log \varphi \right| + 
\frac{\delta^{2}}{2}\cdot \sup_{B}|K \cdot \varphi| 
\leq 
\frac{\delta^{2}\kappa}{2} 
\left[1  + \frac{\pi^{2}}{4} \cdot \left( \frac{\sinh(\delta \sqrt{\kappa}) \tan(\delta \sqrt{\kappa})}{\delta^{2}\kappa}\right)^{2} \right],
$$
and the theorem is proved.
\end{proof}

\section{Appendix}

\begin{lemma}
Suppose that $g:[0,a] \to \RR$ is a continuous strictly convex (resp. concave) function such that $f(x) = g(x)/x$ is a well-defined continuous function. Then $f$ is strictly increasing (resp. decreasing).
\label{lem_concanve_monotone}
\end{lemma}
\begin{proof}
Let $x_{0}\in (0,a)$ and consider the function
$$
h(x) = x \left( f(x)-f(x_{0}) \right) = g(x) - f(x_{0})\cdot x.
$$
The function $h$ is strictly convex (resp. concave) and satisfies $h(0)=h(x_{0})=0$. Thus $h(x) < 0$ (resp. $h(x) > 0$) on $(0,x_{0})$ which implies that $f(x) < f(x_{0})$ (resp. $f(x) > f(x_{0})$) for any $x \in (0,x_{0})$.
\end{proof}
 
\begin{lemma}
We have that 
$
\sinh(x) / x < \exp(x^2/4) 
$
for any $x > 0$, and
$
\sin(x) / x > \exp(-x^2/4) 
$
for any $0 < x < \pi/2$.
\label{lem_some_estimates}
\end{lemma}

\begin{proof}
The first inequality follows directly from the asymptotic expansion of the two functions. Indeed, we have that 
$$
\frac{\sinh(x)}{x} = \sum_{n=0}^{\infty}\frac{x^{2n}}{(2n+1)!}
\qquad \text{and} \qquad
\exp(x^2/4) = \sum_{n=0}^{\infty}\frac{x^{2n}}{4^{n} \cdot n!}.
$$
Since $(2n+1)! > 4^{n} \cdot n!$ for any $n \geq 1$ we see that first inequality holds true. In order to show the second inequality, define $ g(x) = \sin(x)\exp(x^2/4)$. Differentiating twice we obtain that 
$$
g''(x) = \frac{1}{4}\cdot \exp(x^2/4) 
\left[ (x^{2}-2)\sin(x) + 4x\cos(x)\right] .
$$
Let $h(x) = (x^{2}-2)\sin(x) + 4x\cos(x)$. Differentiating the function $h$ twice we obtain that $h''(x) = -(x^2+4)\sin(x)$ so that $h'' < 0$ on $(0,\pi/2)$, which implies that $h$ is strictly concave on $(0,\pi/2)$. Since $h(0) = 0$ and $h(\pi/2) > 0$, we have that $h > 0$ on $(0,\pi/2)$. Thus $g'' > 0$ on $(0,\pi/2)$, and by Lemma \ref{lem_concanve_monotone} we see that $f(x) = g(x)/x$ is strictly increasing. Since $f(0) = 1$ we have that $f(x) > 1$ for any $x \in (0,\pi/2)$, thus completing the proof. 
\end{proof}

\begin{lemma}
For any $a > 0$, the function $x/\tanh(ax)$ is increasing on $(0, \infty)$, and the function $x / \tan(ax)$ is decreasing on $(0, \pi/2a)$.
\label{lem_0048}
\end{lemma}

\begin{proof}
Let $a > 0$. The function $\tanh(ax)$ satisfies $\lim_{x \rightarrow 0} \tanh(ax) / x = a$, and it is strictly concave on $(0, \infty)$. Thus by Lemma \ref{lem_concanve_monotone} we see that $\tanh(ax)/x$ is strictly decreasing on $(0,\infty)$. Since $\tanh(ax) > 0$ on $(0,\infty)$ we obtain that $x / \tanh(ax)$ is increasing.
Similarly, the function $\tan(ax)$ satisfies $\lim_{x \rightarrow 0} \tan(ax)/x = a$, and it is strictly convex on $(0, \pi / 2a)$. 
Thus by Lemma \ref{lem_concanve_monotone} we see that $\tan(ax)/x$ is strictly increasing on $(0, \pi / 2a)$. 
Since $\tan(ax) > 0$ on $(0, \pi / 2a)$ we obtain that $x / \tan(x)$ is decreasing.
\end{proof}

\begin{lemma}
Let $0 < a < \pi/2$. Then the function $f: (0,\pi]\to \RR$ given by
$$
f(x) = \frac{\cos^{-1}(\cos^{2}(a)+\sin^{2}(a)\cos(x))}{x}
$$
is strictly decreasing on $(0,\pi)$, and in particular attains its infimum at $\pi$.
\label{lem_attains_inf}
\end{lemma}

\begin{proof}
Define $g:[0,\pi] \to \RR$ by
$ g(x) = \cos^{-1}(\cos^{2}(a)+\sin^{2}(a)\cos(x)) $.
Then
$$
g''(x) = - \frac{4\sin^{4}(a)\cos^{2}(a)\sin^{4}(x/2)}{(1-(\cos^{2}(a) + \sin^{2}(a)\cos(x))^{2})^{3/2}} < 0,
$$
so that $g$ is strictly concave on $(0,\pi)$.
Moreover, the function $f(x)=g(x)/x$ is well-defined and continuous on $[0,\pi]$. Hence by Lemma \ref{lem_concanve_monotone} we see that $f$ is strictly decreasing.
\end{proof}

\medskip
\noindent Department of Mathematics, Weizmann Institute of Science, Rehovot 76100, Israel. \\
{\it e-mail:} \verb"matan.eilat@weizmann.ac.il"

\end{document}